%% file: geometricADK14.tex
\documentclass[reqno]{jgokova}
\usepackage{epsfig}
\usepackage{graphicx,diagram}
\usepackage{amsmath}
\usepackage{amscd}
\usepackage{graphics}
\usepackage{latexsym}

\input goksty.tex      
\newtheorem{lemm}[thm]{Lemma}
\newtheorem{exam}[thm]{Example}

\newcommand{\eps}{\varepsilon}

\def\CP{\mathcal P}

\def\CF{\mathcal F}

\def\CB{\mathcal B}
\def\C{{\mathbb C}}
\def\CC{\mathcal C}

\def\CL{\mathcal L}
\def\R{{\mathbb R}}
\def\C{{\mathbb C}}

\volume{2}

\begin{document}

\title[Every $4$-Manifold is BLF]{Every $4$-Manifold is BLF}
\author[AKBULUT and KARAKURT]{Selman Akbulut and \c{C}a\u{g}r\i \;Karakurt}

\thanks{The first named author is partially supported by NSF grant DMS 9971440}

\address{Michigan State University Department of Mathematics \\
         East Lansing 48824 MI, USA }
\email{akbulut@math.msu.edu, karakurt@msu.edu}


\begin{abstract}
 Here we show that every compact  smooth $4$-manifold $X$ has a structure of a Broken Lefschetz Fibration (BLF in short). Furthermore, if $b_{2}^{+}(X)> 0$ then it also has a Broken Lefschetz Pencil structure (BLP) with nonempty base locus. This improves a theorem of Auroux, Donaldson and Katzarkov, and  our proof  is topological (i.e. uses $4$-dimensional handlebody theory).
\end{abstract}

\maketitle

\vspace{-.3in}

\section{Introduction}

In order to state our theorem in a precise form, we first need to recall some basic definitions and fix our conventions.

\vspace{.1in}
Let $X$ be an oriented $4$-manifold and $\Sigma$ be an oriented surface. We say that a map $\pi: X \to \Sigma $ has a \emph{Lefschetz singularity} at $p\in X$, if after choosing orientation preserving charts $(\C^2,0) \hookrightarrow (X,p)$ and $\C \hookrightarrow \Sigma $, $\pi$ is represented by the map $(z,w)\mapsto zw$. If  $\pi$ is represented by the map $(z,w)\mapsto z\bar{w}$ we call $p$ an \emph{achiral Lefschetz singularity} of $\pi$.  We say $p\in X$ is a \emph{base point} of $\pi$, if it is represented by the map $(z,w)\mapsto z/w$. 
We call a circle $Z\subset X$ a \emph{fold singularity} of $\pi$, if around each $p \in Z$ we can find charts $(\R \times\R^3,\R)\hookrightarrow (X,Z)$ where $\pi$ is represented by the map
\begin{equation} 
(t,x_1,x_2,x_3)\mapsto (t, x_1^2+x_2^2-x_3^2)
\end{equation}

\defn{Let $X^4$ be a compact oriented $4$-manifold. A \emph{broken Lefschetz fibration} structure on $X$ is a map $\pi :X \to S^2$ which is a submersion in the complement of disjoint union of finitely many points $\CP$ and circles $\CC $ in $X$, such that $\CP$ are Lefschetz singularities of $\pi$, and $\CC$ are fold singularities of $\pi$. When $\CC=\emptyset $, we call $\pi$ a  \emph{Lefschetz fibration}.

\vspace{.05in}

We call $\pi$ a \emph{broken Lefschetz pencil}, if there is a finite set of base points $\CB$ in $X$ (base locus) such that the restriction $\pi: X-\CB\to S^2$ is a broken Lefchetz fibration.  When $\CC=\emptyset $ we call $\pi$ a \emph{Lefschetz pencil}. 

\vspace{.05in}

Furthermore, If we allow achiral Lefschetz singularities among the points of $\CP$  we will use the adjective ``\emph{achiral}'' in front of all these definitions, e.g. \emph{Achiral Broken Lefschetz Fibration}. Sometimes for brevity we will refer them simply by LF, BLF, ABLF, LP, BLP, ABLP.}
Also, by following a convention due to the first author and Ozbagci, we will call a convex Lefschetz fibration (see Definition 1.5) without achiral or fold singularities a \emph{\mbox{Positive} \mbox{Allowable} Lefschetz Fibration}, or PALF for short, if every vanishing cycle is a non-separating curve on fiber.
Now, we can state our main theorem:

\begin{thm}\label{mainthm}
Every closed smooth oriented $4$-manifold $X$ admits a broken Lefschetz \mbox{fibration.} Furthermore if $b_2^+(X)>0$ then $X$ also admits a broken Lefschetz pencil structure (with nonempty base locus).
\end{thm}

\rem{Previously by using analytic methods Auroux, Donaldson and Katzarkov proved that every closed smooth oriented $4$-manifold with  $b_2^+(X)>0$ admits a broken Lefschetz pencil structure \cite{ADK} (so in particular they proved that, after blowing up X admits broken Lefschetz fibration structure). More recently, by using the map singularity theory Baykur \cite{B} and independently Lekili \cite{L} proved that every closed smooth oriented $4$-manifold $X$ admits a broken Lefschetz fibration.  Our proof and  \cite{L}  gives a stronger result, namely it makes the ``fold singularities'' of $\pi$ global, i.e.  around each singular circle $Z$ we may assume that there are charts 
 $(S^1 \times\R^3, S^1)\hookrightarrow (X,Z)$ so that $\pi$ is represented by 
the map (1). 
Recently, in \cite{GK2} there was an attempt by Gay and Kirby to give a topological proof of \cite{ADK},  which ended up producing a weaker version of Theorem 1.2 with ``achiral broken Lefschetz fibration" conclusion. Here we prove this theorem by topological methods, that is  by using $4$-dimensional handlebody techniques. Also, a corollary to our proof is that, $X$ can be decomposed as a union of a convex PALF and a concave BLF glued along their common (open book) boundaries. We hope that the constructive nature of our proof  will be useful for studying $4$-manifolds in the future.}

\vspace{.02in}
 
Our  proof here proceeds by combining the PALF theory of \cite{AO} with the decomposition (positron) technique of \cite{AM}, and it strongly suggests 
linking a certain $4$-dimensional handlebody calculus on ALF's to the Reidemeister moves of the knots in square bridge position. The philosophy here is that,  just as  a PALF $\CF$ is a topological structure which is a primitive of a Stein manifold $W=|\CF|$,  an ALF is a weaker topological structure which is a primitive of an ``Almost Stein manifold''  (for example, while the boundaries of PALF's are positive open books, the boundaries of ALF's are all open books). ALF's are more flexible objects, which are amenable to handle calculus. Our proof uses this flexibility. We hope to study this approach in a more organized way in a future paper.  
\vspace{.05in}

Before proving the main theorem, we remind the reader some basic notions.

\vspace{.05in}

\defn{Let $Y^3$ be a closed oriented $3$-manifold. An \emph{open book decomposition} of $Y$ is a map onto a $2$ dimensional disk, $\pi:Y\to D^2$ such that the preimage of the interior of the disk is a disjoint union of solid tori and the following conditions hold:
\begin{itemize}
	\item $\pi$ is a fibration over $S^1=\partial D^2$ away from these solid tori.
	\item Restriction of $\pi$ to each of these tori is given by the projection map $S^1\times D^2\to D^2$.
\end{itemize} }
 
For all $\lambda \in S^1$ the closure of the set $P_\lambda:=\{\pi^{-1}(r\lambda): r \in (0,1]\}$ is an oriented surface with non-empty boundary which will be referred as a \emph{page} of the open book decomposition. The flow of a pullback of the unit tangent to $S^1=\partial D^2$ defines a diffeomorphism of a fixed page fixing its boundary pointwise. This diffeomorphism is called the \emph{monodromy} of the open book decomposition.

\vspace{.15in}

Now let $X$ be a $4$-manifold with non-empty connected boundary. Let $f:X\to \Sigma$ be a achiral broken Lefschetz fibration where $\Sigma=S^2$ or $D^2$. The following is a description for several boundary behaviors of $f$. Some of the conditions we list here are redundant but we put them for contrasting from each other for a better understanding. 

\vspace{.03in}
\begin{defn}\label{bountype}
We say 
\begin{enumerate}
	\item $f$ is \emph{convex} if
	
	\begin{itemize}
	\item $\Sigma=D^2$.
	\item Fibers of $f$ have non-empty boundary.
	\item $f|_{\partial X}: \partial X\to D^2$ is an open book decomposition of $\partial X$.
	\item[]
\end{itemize}

	\item $f$ is \emph{flat} if
	
\begin{itemize}
	\item $\Sigma=D^2$
	\item Fibers of $f$ are closed.
	\item $f(\partial X)=S^1=\partial D^2$.
	\item $f|_{\partial X}: \partial X \to S^1$ is an honest fibration.
	\item[]
\end{itemize}
	
	\item	$f$ is \emph{concave} if

\begin{itemize}
	\item $\Sigma=S^2=D^2_+\cup D^2_-$ (the upper and lower hemispheres).
	\item Fibers over $D^2_+$ are closed, while  the fibers over $D^2_-$ have non-empty boundary components.
	\item $f(\partial X)\subset D^2_-$.

	\item $f|_{\partial X}: \partial X\to D^2_{-}$ is an open book decomposition of $\partial X$.
	\end{itemize}
\end{enumerate}
\end{defn}
\vspace{0.15in}

\section{Proof of the Main Theorem}

Here for the sake of clarity, we go directly to the proof of our main theorem by using some technical facts, which are proved in the later sections. 

\begin{proof}\textit{of Theorem \ref{mainthm}:}
Our first task is to find a concave fibration inside of $X$. If $b_2^+(X)>0$ a neighborhood of a surface with positive self intersection can be equipped with a concave pencil, see Proposition \ref{concavepen}. Otherwise, we choose any self intersection zero surface in $X$ and find a concave Broken Lefschetz Fibration  on its neighborhood as in Proposition \ref{broconcave}. This concave piece constructed in either way will be called $X_2$. Let $X_1$ be the complement $X_2$. By Lemma 13 in \cite{GK} we may assume that $X_1$ has no $3$ or $4$-handles. In section \ref{stei} below, we proved that every such manifold is ``Almost Stein''. Thus by applying the algorithm in \cite{AO}, we can find a convex Achiral Lefschetz Fibration on $X_1\to D^2$ with bounded fibers. More specifically we get an ALF $\CF$ with total space $|\CF|=X_1$.

\vspace{.05in}

Having decomposed $X=X_1\cup X_2$ as a convex ALF union concave BLF our next aim is to match the open books on the common boundary of these two submanifolds. In section \ref{matchob} below, we described how techniques of \cite{GK2} and \cite{EF} can be modified to achieve this aim. Our version allows us to get the property that  every achiral vanishing cycle (i.e. the framed knot representing it) links twice with an unknot sitting on a page of the open book (coming from $\partial \CF$). This property will be crucial in the next step of our construction. This is the only non explicit step in our proof because our arguments depend on Eliashberg's classification of overtwisted contact structures and Giroux's correspondence between contact structures and open book decompositions.  Note that we may need to introduce extra achiral singularities to the convex side and broken singularities to the concave side while applying Giroux's theorem, in order to match the boundary open books.

\vspace{.05in}

Next, we isotope the unknots mentioned in the previous paragraph to meridians of the binding of the common open book on $\partial X_1=\partial X_2$ and carve the disks bounded by these curves from $X_1$. This amounts to attaching a zero framed $2$-handle to the concave side $X_2$ along a section. Yet another ``\emph{convex-concave decomposition}" of $X$ is obtained this way. Moreover we did not lose the property that boundary open books match. This is the ``positron move'' introduced in \cite{AM}.

\vspace{.05in}

 The final step is what we call the ``\emph{PALFication}'' of the convex side. Now that we have all the defective handles go over a $1$-handle twice all the bad ($tb+1$) framings can be can be turned into a good ($tb-1$) framing by twisting these handles. The twisting operation  induces a self diffeomorphism $\psi$ and a handle decomposition of the convex side $X_1$. The handlebody decomposition is naturally related to a PALF of $X_1$ and $\psi$ sends nonsingular fibers of the previous ALF into fibers of this PALF. In particular, $\psi$ induces a diffeomorphism (isotopic to identity) preserving the open books on the boundary $\partial X_1$. This means that we can glue our new PALF to the concave side and the result is still diffeomorphic to $X$. 
\end{proof}
An easy corollary of this proof is:
\begin{cor}
Every closed smooth oriented $4$-manifold $X$ admits a BLF (or BLP if $b_2^+>0$), by decomposing $X$ as a convex PALF and a concave BLF (or BLP) glued along the common (open book) boundary.
\end{cor}

\section{Concave Fibrations and Pencils}\label{concavesec}
In this section, we illustrate several constructions of concave fibrations and pencils. The major part of these constructions is basically due to Gay and Kirby.

\vspace{.05in}

The easiest example of a concave fibration is obtained by restricting an  Achiral Broken Lefschetz Fibration on a closed manifold to a tubular neighborhood of a regular fiber union a section. Note that the complement of this set is a convex broken Achiral Lefshetz Fibration. More generally, if we remove a convex ABLF from a ABLF (or ABLP) the remaining part is a concave ABLF (or ABLP). The following proposition uses these ideas to construct an interesting family of concave pencils. It first appeared in a draft of \cite{GK2}. Unfortunately, that draft is no longer available. We give a proof of the proposition for completeness.

\vspace{.05in}

\begin{prop}\label{concavepen}
Let $\Sigma$ be a closed oriented surface. Let $N$ be an oriented $D^2$-bundle over $\Sigma$ with positive Euler number. Then, $N$ admits a concave pencil with no broken, achiral or Lefschetz singularities.
\end{prop}
\begin{proof}
Let $\pi:N\to \Sigma$ be the bundle map. Denote its Euler number by $n=e(\pi)>0$. Take a generic section $s:\Sigma \to N$. It intersects with the zero section at $n$ distinct points. Let $D_1,D_2,\cdots,D_n$ be pairwise disjoint small disk neighborhoods of images of these points in $\Sigma$. The section $s$ determines a trivialization $\tau$ of the complementary disk bundle such that the square part of the  following diagram commutes.

\vspace{0.05in}

$\begin{diagram}
 \node{N-\pi^{-1}(D_1\cup \cdots\cup D_n)} \arrow{e,t}{\tau}\arrow{s,l}{\pi} \node{(\Sigma-(D_1\cup \cdots\cup D_n))\times D^2} \arrow{e,t}{p_2}\arrow{s,l}{p_1}\node{D^2=D^2_+}\\
 \node{ \Sigma-(D_1\cup \cdots\cup D_n)}\arrow{e,t}{\text{id}}\node{\Sigma-(D_1\cup \cdots\cup D_n)} 
\end{diagram}$

\vspace{0.05in}

We define the pencil map $f:N\to S^2=D^2_+\cup D^2_-$ by $f(x)=p_2\circ\tau(x)$ for 
$x$ in $N-\pi^{-1}(D_1\cup \cdots\cup D_n)$, and on each $\pi^{-1}(D_i)\cong D^2\times D^2 \subset \mathbb{C}^2$, $f$ is given by the usual pencil map $\mathbb{C}^2\to \mathbb{C}P^1\cong S^2$. It is an easy exercise to verify that these two maps agree on their common boundary and together they define a concave pencil on $N$.

A simpler explanation for this construction can be given via the following handlebody diagrams. We draw the usual handlebody diagram of $N$, disk bundle over a surface $\Sigma$ of genus $g$ with Euler number $n$. The diagram contains one $0$-handle, $2g$ $1$-handles and a $2$-handle with framing $n$ which is attached to the diagram in such a way that the attaching circle passes through each $1$-handle exactly twice. To visualize a concave fibration, we blow up $N$ at $n$ distinct points as in the figure.

\begin{figure}[h]
	\includegraphics[width=.90\textwidth]{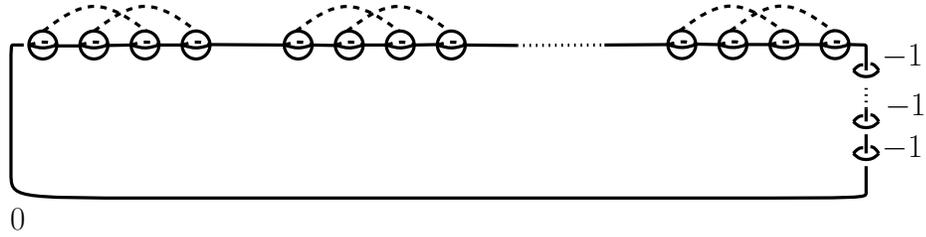}
	\caption{$N\#_n\overline{\mathbb{C}P^2}$}
	\label{fig:concave}
\end{figure}

Forgetting the exceptional spheres for a minute, we see a trivial fibration over disk with fiber genus $g$. To this fibration, we attach several $2$-handles where the attaching circles of the $2$-handle are sections of this fibration restricted to the boundary. By the text in the beginning of this section, the result of attaching these handles is a concave fibration. The boundary open book of this fibration has pages obtained from $\Sigma$ by puncturing at $n$ points. The monodromy of the open book is the product of left handed Dehn twists along curves parallel to boundary components of a page. The last two observations can be verified by the discussion in \cite{EO} or by means of an easy exercise in handle slides:  Anti blow up  at each exceptional sphere. This creates $n$ $0$-framed $2$-handles in the handlebody diagram. Replace each one them with a $1$-handle. The resulting $4$-manifold $X$ admits a convex ALF with the same boundary open book. It is transparent from the handlebody diagram of $X$ that our claims about the boundary open book are correct. Note that if we reverse the orientation of this manifold we get a convex PALF which we can glue to our concave fibration on $N\#_n\overline{\mathbb{C}P^2}$ to get a Lefschetz fibration on the closed manifold $\overline{X}\cup N\#_n\overline{\mathbb{C}P^2}$ where all the singularities lie in the convex side. This Lefschetz fibration is actually symplectic because it admits $n$ sections of self intersection $-1$. Blowing down these sections and isolating $N$, we see that $N$ admits a concave pencil and a compatible symplectic structure $\omega$, making the boundary of $N$ $\omega$-concave.
\end{proof}

\vspace{.05in}

An alternative construction of the concave piece is provided by Gay and Kirby in \cite{GK2} via round handle attachment. We are going to sketch this construction in proposition \ref{broconcave} but before doing that we need to review some basic facts about round handles and their relation with broken fibrations. For details, the reader can consult \cite{GK2} or \cite{B2}. 

\vspace{.05in}

A round $1$-handle is a copy of $S^1\times D^1 \times D^2$ which can be attached to a $4$-manifold along $S^1\times\{-1\}\times D^2\cup S^1\times\{1\}\times D^2$. Up to diffeomorphism, all the attaching information is contained in  a pair of framed links which will be referred as the attaching circles of a round $1$-handle. Moreover, we may choose the framing on one of the attaching circles to be $0$. Every round $1$-handle decomposes uniquely as a union of a $1$-handle and a $2$-handle passing over the $1$-handle geometrically twice algebraically zero times. Therefore, in order to draw a round $1$-handle to a handlebody diagram, it is enough to draw a $1$-handle and a $2$-handle as above. Attaching circles of the round $1$-handle can be visualized by shrinking the legs of the $1$-handle to pairs of points.

\vspace{.05in}
   
Suppose we have an (A)BLF or (A)BLP with boundary which is one of the three types listed in the definition \ref{bountype}. We can attach a round $1$-handle whose attaching circles are sections of the fibration restricted to the boundary. Then, the fibration extends across the round $1$-handle with a fold singularity on $S^1\times\{0\}\times \{0\}$. Resulting object is an (A)BLF or (A)BLP with same type of boundary. Fibers on the boundary, i.e. pages, gain genus. We can visualize this as follows. Attaching circles  of the round $1$-handle  intersect each page at a pair of points. We remove two small disks containing these points  from each page and connect the  boundary circles of these disks by a handle which is contained in the round $1$-handle. The monodromy is composed with some power of a Dehn twist along the boundary circle of one of the disks. This power is precisely the framing of the round $1$-handle.
  
\begin{prop}\label{broconcave}
Let $\Sigma$ be a closed oriented surface. The product $\Sigma\times D^2$ admits a concave Broken Lefschetz Fibration.  
\end{prop}

\begin{proof}
We start with the standard handlebody diagram of $\Sigma\times D^2$. We attach one \linebreak
 $1$-handle and two $2$-handles:  zero framed $\beta_1$ and $-1$ framed $\beta_2$, and a $3$-handle as in Figure~\ref{fig:flattoconcave}. In the figure we did not show the whole diagram. $0$-framed vertical arc is part of the unique $2$-handle of $\Sigma\times D^2$. The framing on $\beta_1$ is chosen so that when we slide it over $\beta_2$ and cancel $\beta_2$ with the $1$-handle, $\beta_1$ becomes a $0$-framed unknot cancelling the $3$-handle. Therefore we did not change the diffeomorphism type of our manifold by these handle attachments.

\begin{figure}[h]
	\includegraphics[width=0.30\textwidth]{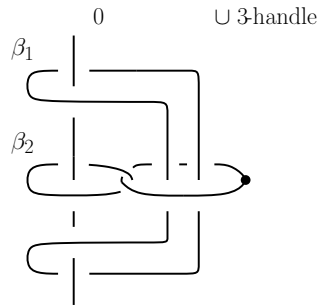}
	\caption{From flat to concave }
	\label{fig:flattoconcave}
\end{figure}

Now we describe how the fibration structure extends across these handles. The \linebreak $1$-handle and $\beta_1$ forms a round $1$-handle which is attached along a bisection of the boundary fibration. As we remarked above, the fibration extends as a broken fibration with one fold singularity inside the round handle. This broken fibration is still flat, i.e. it restricts to a surface bundle over  $S^1$ on boundary. Note that the attaching circle of $\beta_2$ is a section of this surface bundle. Therefore, attaching $\beta_2$ turns the flat broken fibration into a concave broken fibration, creating new disk fibers over the southern hemisphere of the base sphere. Pages are obtained by punching the fibers on boundary along attaching circle of $\beta_2$. Finally, we attach the $3$-handle in such a way that its attaching sphere intersects each page along a properly embedded arc $\gamma$. Concave fibration extends across the $3$-handle without forming any new singularities. The fibers over the southern hemisphere gain a $1$-handle whereas pages lose a $1$-handle (they are cut open along the arc  $\gamma$).
\end{proof}

The last topic we would like to discuss in this section is stabilization of concave \linebreak
fibrations. Recall that positive (resp. negative) stabilization of an open book amounts to plumbing a left (resp. right) handed Hopf band to its pages and composing its monodromy by a right handed (resp. left handed) Dehn twist along the core of the Hopf band. If the open book is the boundary of a convex ALF positive (resp. negative) stabilization can be achieved by creating a canceling $1$-$2$ handle pair where  attaching balls of the \linebreak
$1$-handle are strung on the binding and the $2$-handle lies on a page with framing one less (resp. one more) than the framing induced by the page. Stabilization of concave fibrations is not that straight forward because $1$-handles cannot be attached this way. Instead we create an extra cancelling $2$-$3$ handle pair. Moving the new $2$-handle over the $1$-handle creates a round $1$-handle which can be attached to concave fibration, see Figure \ref{fig:stabconcave}. It is evident from the picture that the resulting concave fibration has a new broken singularity, and a Lefschetz (resp. achiral) singularity inside.

\begin{figure}[h]
	\includegraphics[width=0.30\textwidth]{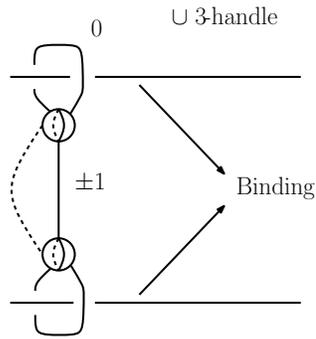}
	\caption{Stabilization of concave fibrations}
	\label{fig:stabconcave}
\end{figure}

It is possible to negatively stabilize concave fibrations without introducing achiral singularities inside. The trick goes as follows; we create another cancelling  $1$-$2$ handle pair, where the $2$-handle has a good (-1) framing. We move bad (+1) framed $2$-handle over the new $1$-handle and regard them as a round $1$-handle, see Figure \ref{fig:goodstab}.

\begin{figure}[h]
	\centering
		\includegraphics[width=0.44\textwidth]{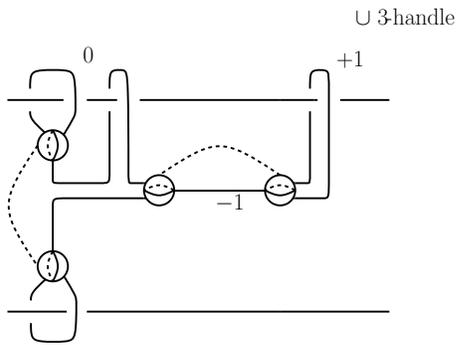}
	\caption{Negative stabilization without achiral singularity}
	\label{fig:goodstab}
\end{figure}

We formed two fold singularities and a Lefschetz singularity inside of the concave fibration. One right handed and two left handed Hopf bands are plumbed to the boundary open book as in Figure 9 of \cite{GK2}.

\section{Stein Manifolds}\label{stei}
Here, we review some basic facts about Stein manifolds of real dimension $4$. Afterwards, we will discuss their natural generalizations. For more discussion about Stein manifolds, reader can consult \cite{OS}.

\vspace{.05in}

Let $X$ be a complex manifold, and let  $J:TX\to TX$ be the corresponding almost complex structure, with its dual $J^*:T^*X\to T^*X$. To any function $\varphi : X\to \mathbb{R}$ we can associate a $2$-form $\omega_\varphi:=-dJ^*d\varphi$ and a symmetric $2$-tensor $g_\varphi:=\omega_\varphi(\cdot,J\cdot)$. A function $\varphi$ as above is called \emph{strictly plurisubharmonic} if $g_\varphi$ is a Riemannian metric. A compact complex manifold $X$ with non-empty connected boundary is called a \emph{Stein manifold} if it admits a strictly plurisubharmonic function which is constant on its boundary. In this case, the pair $(J,\varphi)$ is called a \emph{Stein structure} on $X$.

\vspace{.05in}

Stein manifolds come equipped with many extra structures. The exact form $\omega_\varphi$ is non-degenerate, so it defines a a symplectic structure. The gradient vector field $\nabla \varphi$ of the strictly plurisubharmonic function is a symplectic dilation, i.e. the Lie derivative preserves the symplectic form $\CL_{\nabla\varphi}(\omega_\varphi )=\omega_\varphi$, and it points outwards the boundary. This implies that the set of complex tangencies of the boundary form a contact structure on the boundary.  
In \cite{E}, Eliashberg gave the topological characterization of $4$-dimensional Stein manifolds.
\begin{thm}\label{eli}(Eliashberg)
Let $X=B^4\cup(1\text{-handles})\cup (2\text{-handles})$ be a  $4$-dimensional handlebody with one $0$-handle and no $3$ or $4$-handles. Then
\begin{enumerate}
\item the standard Stein structure on $B^4$ can be extended over the $1$-handles so that the manifold $X_1:=B^4\cup (1\text{-handles})$ is Stein.
\item If each $2$-handle is attached to $\partial X_1$ along a Legendrian knot with framing one less than the Thurston-Bennequin framing then the complex structure on $X_1$ can be extended over the $2$-handles  making $X$ a Stein manifold.
\item The handle decomposition of $X$ is induced by a strictly plurisubharmonic function.
\end{enumerate} 
\end{thm}
By abuse of language, we call a handle decomposition as in Theorem \ref{eli} as a Stein structure without specifying the almost complex structure or strictly plurisubharmonic function. These two objects should be understood to be constructed as in the proof in \cite{E}. Now by
using the characterization in Theorem \ref{eli} we give the following definition to generalize the Stein manifolds to a larger category. 
\begin{defn}
Let $X$ be a compact oriented $4-$manifold. X is said to be \emph{almost Stein} if it admits a proper Morse function $f:X\to\mathbb{R}$ with the following three properties:
\begin{enumerate}
	\item Critical points of $f$ have index no bigger than $2$,
	\item Boundary $\partial X$ is a level set for $f$,
	\item Attaching circles of $2-$handles associated to $f$ are isotopic to some Legendrian curves in $(\partial X^{(1)},\xi_0)$ such that the framing on any handle is either one less than or one greater than the Thurston-Bennequin framing. Here, $X^{(1)}$ is the union of $0$ and $1-$handles of $X$ and $\xi_0$ is the unique Stein fillable contact structure on its boundary. 
\end{enumerate}
\end{defn}

In Theorem \ref{alstein} below we show that almost Stein manifolds exist in abudence. The same theorem, stated in different terms was proved in \cite{GK}. In an almost Stein manifold the $2-$handles, which are attached with framing $tb-1$, will be called \emph{good handles}, otherwise they will be called \emph{defective}. 

\begin{thm}\label{alstein} Every $4$-dimensional handlebody $X$ with without $3$ and $4$-handles is an almost Stein manifold with at most one defective handle.
\end{thm}

\begin{proof}
Without loss of generality we can assume that $X$ has one $0$-handle. By the first part in theorem \ref{eli}, $X_1=B^4\cup (1\text{-handles})$ has a Stein structure. We must show that $2$-handles can be arranged to be attached with correct framings. Let $h$ be a $2$-handle in $X$. First, we isotope the attaching circle $K$ of $h$ to a Legendrian knot. Say, the framing on $K$ is $n$ with respect to the Thurston-Bennequin framing. The number $n$ can be increased arbitrarily by adding zigzags to the front projection of $K$. Decreasing $n$ is tricky. Near a local minimum of the front projection of $K$, we make Legendrian Reidemeister~1 move, see \cite{G}. This operation does not change the Thurston-Bennequin framing. Next, we create a pair of cancelling $1$ and $2$-handles, where the $2$-handle has framing $+1$ as in the left hand side of the Figure \ref{alsteinp}. Sliding $h$ over the new $2$-handle, we can decrease its framing. Finally, we observe that the same cancelling pair can be used to fix the framing of all $2$-handles, leaving only one defective handle as desired.  
\begin{figure}[h]
	\includegraphics[width=0.60\textwidth]{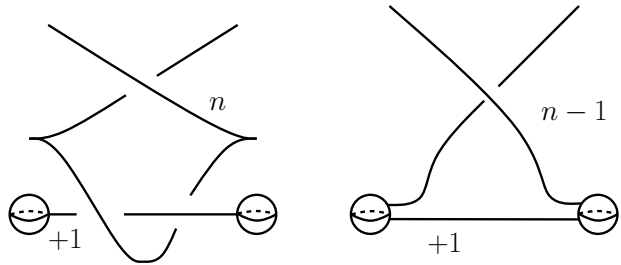}
	\caption{Fixing framings}
	\label{alsteinp}
\end{figure}

\end{proof}

\vspace{.05in}

\begin{rem}

\begin{enumerate}
\item Creating a pair of $1$-handle and a $+1$ framed $2$-handle corresponds to negative stabilization of the contact structure $\xi_0$ on $\partial X_1$. This creates an overtwisted disk, and sliding a Legendrian knot over that disk increases its Thurston-Bennequin invariant. This observation was used by Gay and Kirby to prove Theorem \ref{alstein}. The only significant difference is that they apply Lutz twists to create an overtwisted disk.

\vspace{.02in}

\item Yet another proof for Theorem \ref{alstein} via \cite{AO} can be given via the correspondence between achiral  Lefschetz fibrations and almost Stein structures which we state more precisely in Section 6 below. Harer \cite{H} (see also \cite{EF}) proved that every $4$-dimensional handlebody $X$ without any $3$ or $4$-handles admits an achiral Lefschetz fibration over the disk with bounded fibers. This fibration naturally makes $X$ an almost Stein manifold.
\end{enumerate}
\end{rem}

The main theorem in this section asserts that by carving certain properly embedded disk we can transform defective handles into good ones. The following example provides a model for that transformation.

\vspace{.05in}

\begin{exam}\label{deffix} Let $X=D^2\times S^2$. By [E], $X$ does not admit a Stein structure. The simplest handlebody decomposition of $X$ consists of one $0$ and one $2-$handle, where the $2-$handle is attached to an unknot with zero framing, as in Figure \ref{fig:legamend}. The obvious Legendrian realization of this unknot has front projection with two cusps. Thurston-Bennequin number of this Legendrian knot is $-1$ so, the $2-$handle is attached with framing $tb+1$.  An alternative Legendrian realization of the unknot is drawn next to it. The curve $\gamma$ is not a part of the handlebody yet. We will show that the manifold $\widetilde{X}=X-D$ is Stein where $D$ is a properly embedded disk in $0-$handle whose boundary is $\gamma$.

\vspace{.02in}

Consider the Stein manifold on the right hand side of Figure \ref{fig:legamend}. It is an easy exercise to show this manifold is diffeomorphic to $\widetilde{X}$. Indeed, we can obtain a handlebody decomposition of $\widetilde{X}$ by putting a dot on $\gamma$  (i.e. turning it into a $1$-handle) in the previous figure. Create a cancelling $1$ and $2$ handle pair,  where the two handle has framing $-1$. Slide the new $1$-handle over the $1$-handle $\gamma$. Next slide two strands of original $2$-handle over the new $2$-handle, and then cancel the $1$ and $2-$handle pair. The resulting manifold is the smooth manifold decribed by the next figure.
\begin{figure}[htbp]
	\includegraphics[width=0.90\textwidth]{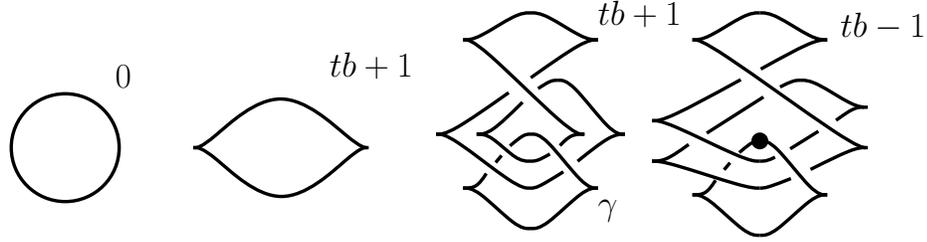}
	\caption{Repairing defective handles in an almost Stein manifold}
	\label{fig:legamend}
\end{figure}

\end{exam}

Applying the procedure  described in the Example \ref{deffix} to each defective handle in an almost Stein manifold we arrive at the following conclusion.
\begin{thm}
Every almost Stein manifold admits a Stein structure in the complement of finitely many properly imbedded disks.
\end{thm}

By attaching a good handle and homotopically cancelling this $1$-handle created in the procedure, we see that \emph{Steinification} of every almost Stein manifold can be made within its homotopy class. This was first observed by the first author and Matveyev in \cite{AM} (e.g. positrons). 

\begin{thm}(Akbulut-Matveyev)
Every almost Stein manifold is homotopy equivalent to a Stein manifold.
\end{thm}

Almost Stein manifolds do not carry as many extra structures as Stein manifolds. For example, they do not naturally admit a symplectic structure. An almost complex structure can be defined on an almost Stein manifold, but only in the complement of finitely many points each contained in a $2$-handle. On the other hand, we still have a contact structure on the boundary defined by means of the set of complex tangencies. This contact structure will be denoted by $\xi_X$ for an almost Stein manifold $X$.

The handlebody diagram inducing an almost Stein structure on $X$ can easily be converted into a contact surgery diagram in the sense of Ding and Geiges \cite{DG}. We replace each $1$-handle with a circle with dot, a convenient notation first introduced in \cite{A} (which means that the interior of the disk, which this circle bounds, is pushed into the  $4$-manifold interior, then its tubular neighborhood is removed). The only additional assumption is that we choose the circle to be a Legendrian unknot, and we put a negative twist to each $2$-handle that passes over the $1$-handle  as in Figure \ref{fig:legamend}. The twisting operation is done in order to match the conventions to calculate the relative framing with respect to Thurston-Bennequin framing. We convert our Legendrian handlebody diagram into a contact surgery diagram by erasing the dots on $1$-handles and regarding them as \linebreak
$(tb+1)$-framed $2$-handles. Topologically, this operation corresponds to surgering some circles, but it does not affect the boundary contact manifold.  
  
Homotopy invariants of $\xi_X$ can easily be read off from the front projection of the Legendrian handlebody diagram inducing the almost Stein structure of $X$, see \cite{G}. The Poincar\'e dual of its first Chern class is represented by a chain $\sum k_i[\mu_i]$, where  $k_i$ is the rotation number of a Legendrian $2$-handle and $[\mu_i]$ is the homology class represented by the meridian of that handle. On $\partial X$, the oriented $2$-plane fields having the same first chern class $c_1$ is a $\mathbb{Z}/d$ torsor, where $d$ is the divisibility of $c_1$. Each element of this torsor can be distinguished by its $3$-dimensional invariant $d_3=(c_1(X)^2-3\sigma(X)-2 \chi(X))/4 +q$, provided $c_1$ is torsion. Here $q$ is the number of defective handles, $c_1(X)$ is a cohomology class such that $c_1(X)(h)=\text{rot}(L)$ for every $2$-handle $h$ with Legendrian attaching circle $L$, and the cup square of $c_1(X)$ is taken in some appropriate sense. Let us call the operation of creating a cancelling pair of a $1$-handle and a $(tb+1)-$framed $2$-handle by \emph{negative stabilization}. This operation does not have any topological effect but it does change the almost Stein structure on $X$ and the contact structure on $\partial X$. The key observation here is that if we have two almost Stein manifolds with same first Chern class of the contact structures on the boundary, then we can match the contact structures up to homotopy by negatively stabilizing one of the almost Stein manifolds.  

\section{Destabilizations}

This section is devoted to a technical lemma. It will be used to determine the monodromy of a $(p,q)$ torus knot but we believe that the lemma is important on its own as it provides a sufficient condition to destabilize an open book.

\begin{lemm}\label{desta}
Let $S$ be a page of an open book decomposition of the $3$-sphere $S^3$. Suppose that $K$ is a non-separating knot on $S$ with self linking number $-1$ with respect to the framing induced by $S$. Then $S$ is isotopic rel $K$ to the plumbing of some surface $\widetilde{S}$ with a Hopf band $H$ whose core is $K$. Moreover, the surface  $\widetilde{S}$ is isotopic to $S$ cut along a properly embedded arc $\gamma$ that intersects transversely with $K$ at precisely one point.
\end{lemm}

\begin{proof}
Since $S$ has nonempty boundary and $K$ is non-separating, the arc $\gamma$ in the statement exists. Create two parallel copies of $K$ on $S$, one to the left of  $K$ and other to the right. Call them $K_1$ and $K_2$ respectively. The union of these two curves is the boundary of a Hopf band $H$ whose core is $K$. 

In Figure \ref{destabi} below we describe an isotopy of $S$ in which it is apparent that $H$ is plumbed to the rest of the surface. Starting at a little right of one of the components of $\gamma\cap \partial S$ we perform a finger move along a parallel copy of $\gamma$ till we hit $K_1$. Next we turn right and continue the finger move along $K_1$ until we approach $\gamma$ from its left. Finally, we perform the same isotopy from the other side but this time we turn left at $K_2$. 

This isotopy isolates $H$ from $S$. We can write $S=H\cup \widetilde{S}$ where the intersection $H\cap \widetilde{S}$ is a small neighborhood of the arc $\gamma\cap H$ in $H$. This is precisely the definition of plumbing.

To see why the last statement holds, we cut $H$ along an arc from $K_1$ to $K_2$ intersecting with $K$ exactly once. Clearly, we deplumbed $H$ from $\widetilde{S}$. Now, pulling everything back, the assertion should be clear. 
\begin{figure}[h] 
	\includegraphics[width=0.50\textwidth]{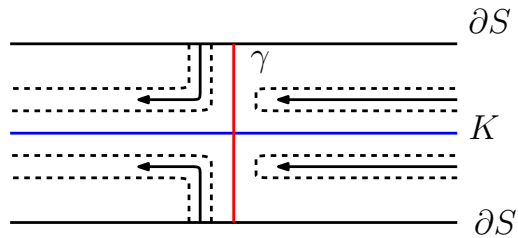}
	\caption{Isotopy by finger moves. Neighborhood of the blue curve is a Hopf band plumbed to the rest of the surface.  }
	\label{destabi}
\end{figure}
    
\end{proof}

\section{ALF-Almost Stein}\label{alfol}

In this section, we are going to review the algorithm of Akbulut and Ozbagci to establish a one to one correspondence between ALFs and almost Stein manifolds. Originally, in \cite{AO} authors established this correspondence between PALFs and Stein manifolds but their arguments hold in our case as well.

\begin{thm}\label{akoz}(Akbulut-Ozbagci)
A $4$-manifold is almost Stein if and only if it admits an ALF over the disk with bounded fibers.
\end{thm}   

Combining this result with our previous observation, Theorem \ref{alstein}, we get the Harer's theorem \cite{H}:

\begin{cor}
Every $4$-dimensional handlebody without $3$ or $4$ handles admits an ALF over disk with fibers having non-empty boundary.
\end{cor}

\begin{proof}[Proof of Theorem \ref{akoz}] 

\vspace{.05in}

We will prove almost Stein implies ALF part only. First, we prove a special case where $X$ has no $1$-handles. By assumption, $X$ is obtained by attaching several $2$-handles along a Legendrian link $L$ in $S^3$ with standard contact structure. In the front projection of $L$, we replace all local maxima and minima with corners  and then turn the picture $45^\circ$ counter clockwise to get the square bridge diagram of $L$. This diagram allows us to find a $(p,q)$ torus knot $K$ and a Seifert surface $\Sigma$ of $K$ with the following properties:
\vspace{.05in}
\begin{enumerate}
\item\label{iteaone} Complement of $K$ fibers over $S^1$ with fibers $\Sigma$, forming an open book decomposition of $S^3$, which  supports the standard contact structure of $S^3$,
\item\label{iteatwo} The open book decomposition in item (\ref{iteaone}) extends to a PALF of $4$-ball $B^4$ with $(p-1)(q-1)$ vanishing cycles,
\item The link $L$ can be put on $\Sigma$ such that the surface framing agrees with the Thurston-Bennequin framing.
\end{enumerate}

Once we achieve to find such $K$ and $\Sigma$, PALF of $B^4$ in item (\ref{iteatwo}) extends across the $2$-handles as an ALF. Fibers are still diffeomorphic to $\Sigma$ but the monodromy is composed with product of Dehn twists along components of $L$. Dehn twists are right handed for good handles and left handed for defective ones. 

\vspace{.05in}

The original construction of $K$ and $\Sigma$ is due to Lyon \cite{Ly}. Without loss of generality, we may assume that $L$ is contained in the cube $[-1,1]\times[0,1]\times[0,1]\subset \mathbb{R}^3\subset S^3$. Say, $L$ has $p$ horizontal and $q$ vertical segments. Choose numbers $1>a_1>a_2>\cdots a_p>0$ , and $0<b_1<b_2<\cdots<b_q<1$, and a small constant $\eps >0$.  For the $i$-th horizontal segment of $L$, we put the plate $\{1\}\times[0,1]\times [a_i-\eps,a_i+\eps]$. For the $j$-th vertical segment, we put the plate $\{-1\}\times [b_j-\eps,b_j+\eps]\times [0,1]$. We connect horizontal segments to vertical ones by a band making quarter twist. The result is a Seifert surface of a $(p,q)$ torus knot with minimal genus. 

\vspace{.05in}

Clearly, we can put $L$ on the surface $\Sigma$ constructed in this way. It is an easy exercise to verify that page framing on each component of $L$ is its writhe minus the number of \emph{south-west} corners. This is exactly the recipe for Thurston-Bennequin framing. If we project $\Sigma$ to $yz-$plane, we get a $(p,q)$-grid. This grid has $(p-1)(q-1)$ holes in it. Take a loop enclosing one of these holes inside the grid and lift it to $\Sigma$. This curve has self linking number $-1$ with respect to the page framing. Hence, we can destabilize the Hopf band around the curve as in lemma \ref{desta}, leaving $(p-1)(q-1)-1$ holes in the grid. If we continue to destabilize, we can verify (\ref{iteatwo}) by explicitly finding the monodromy curves.

\vspace{.05in}

Finally, we saw that the monodromy of $(p,q)$ torus knot is a product of right handed Dehn twists. Therefore, the contact structure it supports is Stein fillable. Eliashberg \cite{E2} proved that every such contact structure on $S^3$ is isotopic to the standard one. 

\begin{rem}
One problem in contact topology is to find an explicit open book decomposition supporting a given contact structure starting from its contact surgery diagram. As noticed by Stipsicz \cite{St}, the algorithm we described above provides an open book decomposition supporting a contact structure given by a surgery diagram if the framing coefficients are $\pm 1$. Ozbagci extended this result for the case of rational surgery in \cite{O}. The idea of destabilizing the unnecessary Hopf bands in this algorithm is used by Arikan in \cite{Ar} to find a smaller genus open book supporting the given contact structure.
\end{rem}

We continue the proof by considering the general case. First, convert all the $1$-handles into circles with dots as described at the end of section \ref{stei}. Ignoring the dots for a minute, we apply our algorithm to this  new Legendrian link we obtained to put it on a page of a $(p,q)$ torus knot. Each component of the Legendrian link is a non-separating curve on a page. By an isotopy we can put those components which represent $1$-handles in such a position that each one of them intersects every page exactly once. We push the disks bounded by these circles inside of $4$-ball, each one of them is realized as a section, then scoop out tubular neighborhoods of these sections. The last operation creates holes on the fibers. Of course, the number of these holes is the same as the number of $1$-handles.  
\end{proof}

As an illustration, we apply the algorithm described in the proof to get ALF pictures of the last two Legendrian handlebody diagrams in Figure \ref{fig:legamend}. In Figure \ref{fig:exalf1} we put both diagrams on Seifert surface of a $(7,7)$ torus link. The black curves represent $1$-handles. That is why they are clasped to the binding. The monodromy of the ALF on the top picture,  is a left handed Dehn twist along the red curve composed with the monodromy of the torus knot; whereas the monodromy of the PALF on the bottom picture, is a right handed Dehn twist along the
blue curve composed with the monodromy of the torus link.  

\newpage

\begin{figure}[h]
	\includegraphics[width=0.54\textwidth]{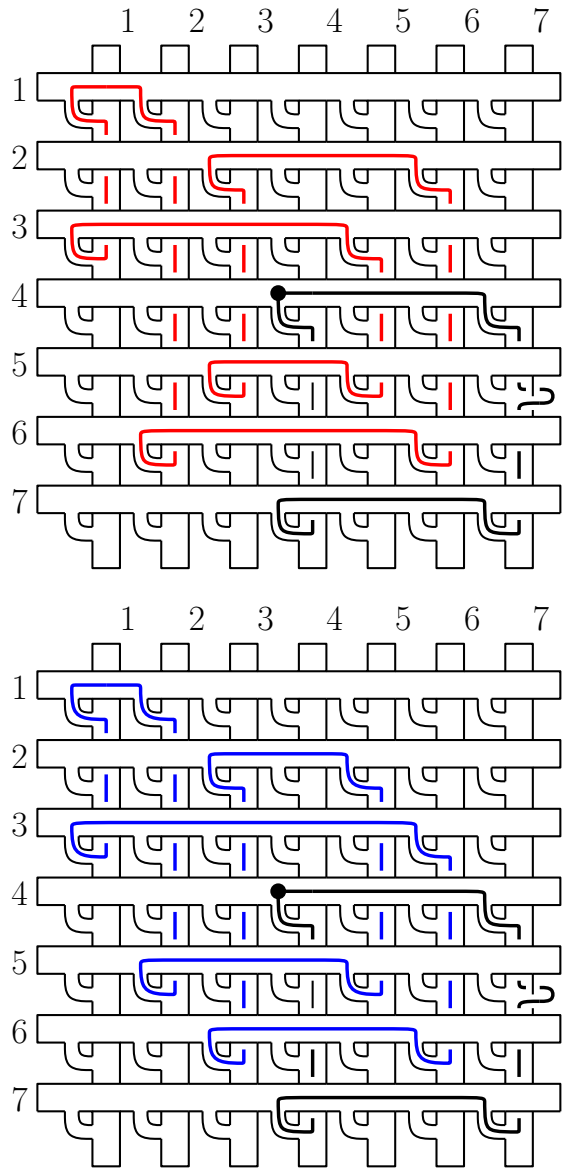}
	\caption{Two ALFs}
	\label{fig:exalf1}
\end{figure}

\newpage

\begin{rem}
A quick way of visualizing how the ``circles with dots'' in Figure 8 bound disks intersecting each page of the ALF once is as follows (c.f. \cite{AO}): First draw the handlebody picture of the ALF, for example the first picture of  Figure 9. Note that since the self linking of any circle with dot $C$ is $-1$, by the destabilization move of Lemma 5.1, we can assume the position of $C$ as in the picture. Now after the indicated isotopies, $C$ looks like in second picture of Figure 9, where it obviously bounds a disk intersecting each page of the ALF once (note the change of framing of $C$  in this process).

\end{rem}
\begin{figure}[h]
	\includegraphics[width=0.6\textwidth]{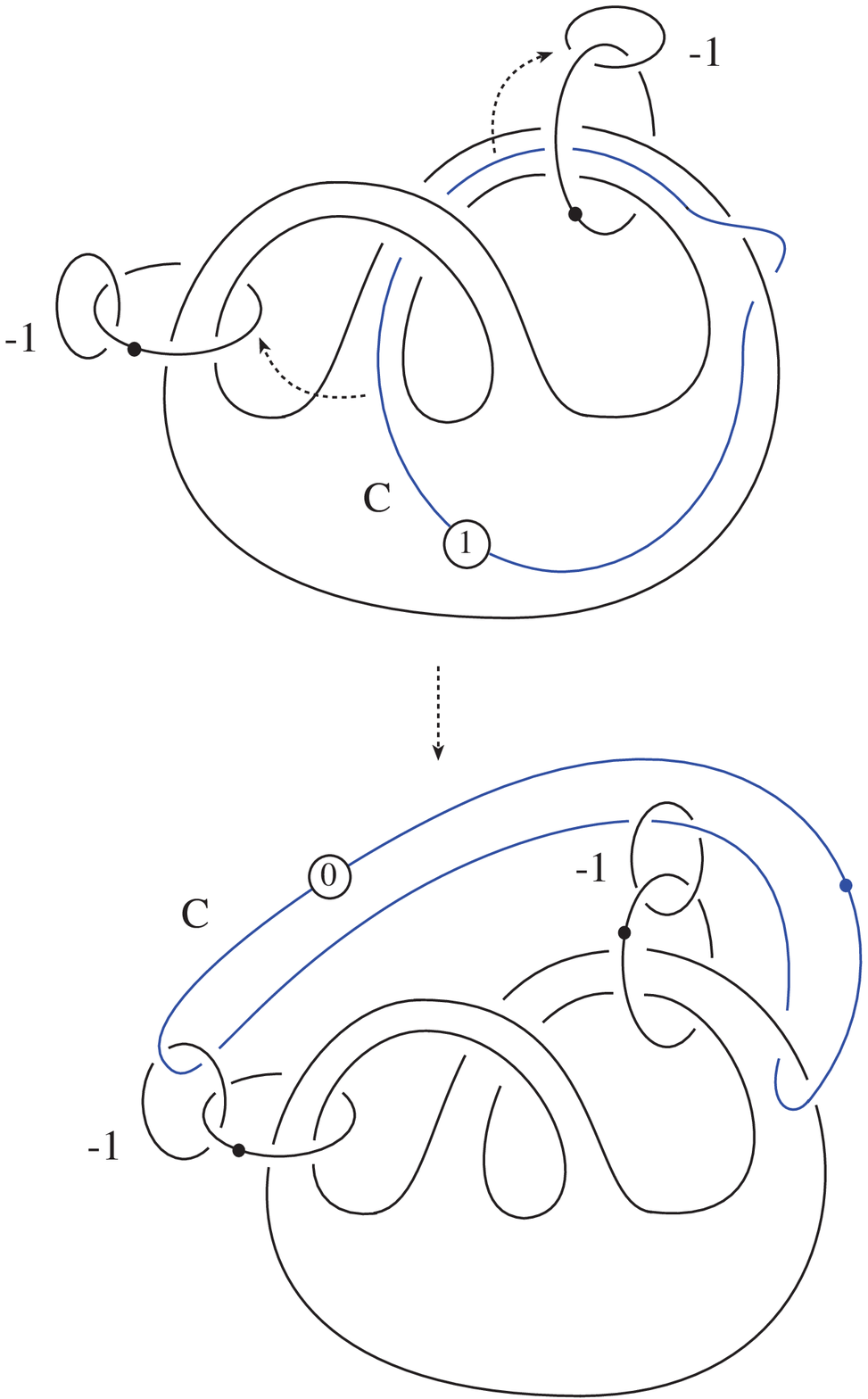}
	\caption{Carving}
	\label{fig:carving}
\end{figure}

\section{Some ALF Moves and Their Effects}\label{movesec}

In this section, we are going to describe a set of moves that produce new ALF out of a given one. We will describe the moves by means of a handlebody diagram in which $2$-handles are drawn in square-bridge position. The corresponding ALF should be understood to be constructed as in Section \ref{alfol}. It is our hope that these moves can be extended to result in a uniqueness theorem for Lefschetz type fibrations.

\vspace{.05in}

\textbf{Move 0:}(\emph{Refining}) Our first move is increasing the number of grids on fibers constructed as in section \ref{alfol}. It amounts to positively stabilizing the initial ALF several times. Most of the forthcoming moves require refining before we perform them. We are going to omit saying that we apply this move when it is clear from the context what sort of refining must be done.

\vspace{.05in}

\textbf{Move 1:}(\emph{Isotopy}) Any sort of isotopy in the category of square-bridge knots is applicable provided that it does not change the Thurston-Bennequin framing. This includes \emph{square-bridge Reidemeister moves} which we show in Figure \ref{fig:R1}.

\begin{figure}[h]
	\includegraphics[width=0.40\textwidth]{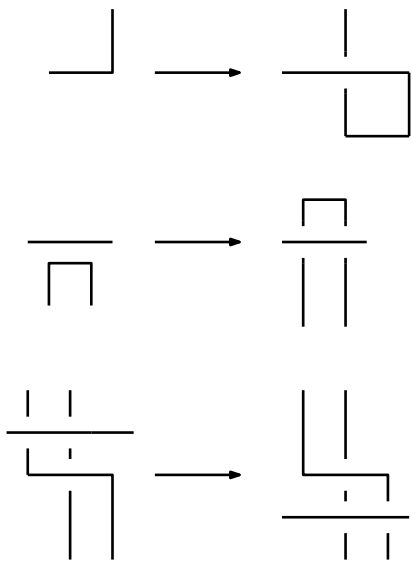}
	\caption{Square-bridge Reidemeister moves}
	\label{fig:R1}
\end{figure}

\textbf{Move 2:}(\emph{Stabilization}) We include both positive and negative stabilizations in our list of moves. As mentioned above, increasing  the grid is a special type of positive stabilization. In general, we indicate positive (resp. negative) stabilizations by creating a pair of canceling $1$-handle and a $2$-handle where the legs of the $1$-handle are strung on the binding and the $2$-handle has framing $tb-1$ (resp. $tb+1$).

\vspace{.05in}

\textbf{Move 3:}(\emph{Finger Move}) Let $K$ be an attaching circle of a $2$-handle lying on a page with framing either one less or one more than the framing induced by the pages. Near $K$, we stabilize our ALF and make $K$ run over the $1$-handle of the stabilization without changing its framing, see Figure \ref{fig:fingeralt2}. This move is called positive or negative finger move depending on what type of stabilization we made. Performing one positive and one negative finger move consecutively does not change the total space of the ALF. However, it adds or subtracts $2$ to the number obtained by evaluating the first Chern class of the boundary contact structure on the meridian of $K$ depending upon the orientation (see \cite{EF} Lemma 10). Positive finger move has a nice interpretation in terms of square-bridge knots if the grid is refined enough. Given any horizontal or vertical segment of $K$, we create a small rectangle on top of it as in Figure \ref{fig:finger}, then install this to our ALF. Clearly, this adds  one negative twist to the page framing as we traverse the rectangle. Since the monodromy of a $(p,q)$ torus link is a product of right handed Dehn twists only, there is no similar interpretation for left finger move in our language. 

\vspace{.1in}

\begin{figure}[h]
	\includegraphics[width=0.70\textwidth]{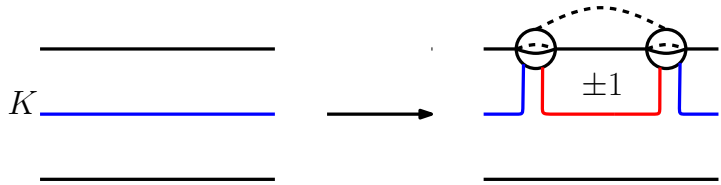}
	\caption{Positive and negative finger moves}
	\label{fig:fingeralt2}
\end{figure}

\begin{figure}[h]
	\includegraphics[width=0.49\textwidth]{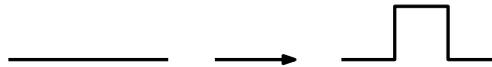}
	\caption{Another interpretation of positive finger move}
	\label{fig:finger}
\end{figure}

\vspace{.1in}

\section{Matching the Open Books}\label{matchob}

In this section, we are going to prove a technical lemma to match open books of a convex ALF and a concave BLF.

\begin{lem}

Let $X=X_1\cup X_2$ where $X_1$ is a convex ALF and $X_2$ is a concave BLF constructed as in either proposition \ref{concavepen} or \ref{broconcave}. Assume that the first homology group of $\partial X_1=\partial X_2$ has no $2$-torsion. Then Lefschetz type fibrations on $X_1$ and $X_2$ can be modified such that the boundary open books match. Moreover, we can arrange that every achiral vanishing cycle on the new convex ALF links twice with an unknot on a page. 
\end{lem}

\begin{proof}
Denote the induced open books on $\partial X_1$ and $\partial X_2$ be $\it{ob}_1$ and $\it{ob}_2$. The contact structures supported by these two open books will be denoted by $\xi_1$ and $\xi_2$. By a theorem of Giroux \cite{Gi}, we can match the open books up to positive stabilization if we can match the associated contact structures.

\begin{thm}(Giroux)
Two open book decompositions of a $3$-manifold have common positive stabilizations if and only if they support isotopic contact structures.
\end{thm}

At this point, we would like to explain what we mean by matching the open books. Boundaries of convex and concave sides are identified by an orientation reversing diffeomorphism $f:\partial X_2 \to \partial X_1$. We apply Giroux's theorem to the open books $f^*\it{ob}_1$ and $ob_2$ to conclude that they have common positive stabilization if and only if the contact structures $f^*\xi_2$ and $\xi_2$ are isotopic. Note that stabilizing $\it{ob}_2$ requires a round handle \linebreak attachment  as in section \ref{concavesec}.    Consequently, new broken singularities are formed in the concave side. In addition, positive stabilization of $f^*\it{ob}_1$ means that we need to negatively stabilize $\it{ob}_1$, so, new achiral singularities are formed in the convex side. 

\vspace{.05in}

Now, we consider the problem of matching the contact structures $f^*\xi_1$ and $\xi_2$.   Before matching them up to isotopy, we study a more elementary problem, namely, matching them up to homotopy. To simplify the notation we drop $f^*$ from $f^*\xi_1$ and $f^*\it{ob}_1$. As explained in \cite{G}, two oriented $2$-plane distributions on a $3$-manifold are homotopic if and only if they have the same Chern class and the $3$-dimensional invariant  $d_3$, in case the first homology has no $2$-torsion.
\vspace{.05in}

We can make the Chern class of $\xi_1$ equal to $c_1(\xi_2)$ using techniques of Section \ref{movesec}. The Poincar\'e dual of $c_1(\xi_1)$ is represented by a cycle $\sum k_i[\gamma_i]$ where $\gamma_i$ is a meridian of a $2$-handle. The handle could be good or defective but we know that its attaching circle is in square-bridge position. The number $k_i$ is the number of up corners minus the number of down corners. We may increase or decrease this number by $2$ applying a positive and a negative finger move consecutively. It can be shown that in our examples $c_1(\xi_2)=0$. Therefore we are done with matching Chern classes provided that $c_1(\xi_1)$ has correct parity. This is true because $c_1(\xi)$ is even for an arbitrary oriented $2$-plane distribution on a $3$-manifold $M$. To see this, fix a trivialization of the tangent bundle. Unit normal vector field to $\xi$ defines a map $g:M\to S^2$ such that $\xi=g^*TS^2$. So $c_1(\xi)=g^*c_1(TS^2)=2g^*PD([\text{point}])$.

\vspace{.05in}

To equate the $3$-dimensional invariants, we negatively stabilize one of $\it{ob}_1$ and $\it{ob}_2$ enough number of times. If we want to avoid achiral singularities in the concave side, we need to form an extra round handle singularity and perform two positive stabilizations on $\it{ob}_2$.

\vspace{.05in}

Finally, we apply a theorem due to Eliashberg \cite{E3}, which says that two overtwisted contact structures on a $3$-manifold are homotopic if and only if they are isotopic. To guarantee that both $\xi_1$ and $\xi_2$ are overtwisted we negatively stabilize both their open books once more.

\end{proof}
\section{ALF Repairing or PALFication}

This section is devoted to a procedure of getting rid of achiral singularities of a convex ALF without changing the total space and the boundary open book. For the sake of clarity, here rather than proving the main result in generalities, we will explain the proof on an example.

\begin{exam}
We first describe our initial ALF. Consider the $4$-ball $B^4$ fibering over disk with disk fibers. This fibration induces  an open book decomposition of the boundary $\partial B^4= S^3=\mathbb{R}^3\cup \{\infty\}$. To visualize this open book, we indicate the binding as an unknot drawn as a rectangle in $\mathbb{R}^3$. Then, the meridian of this rectangle is a section of a fiber bundle over circle with disk fibers. Next, we negatively stabilize this open book to get a convex ALF with one defective handle. Recall that negative stabilization amounts to creating a pair of cancelling $1$ and $2$-handles where the feet of $1$-handles are strung on the binding and the attaching circle of the $2$-handle is on a page with framing one more than the surface framing. The main theorem requires a defective handle to pass twice over a $1$-handle. This is why we change the picture by isotopy and carve a $4$-ball in such a way that the $2$-handle satisfies the assumption of the main theorem. The square-bridge position of the final handlebody is given in Figure \ref{fig:repairsq} on the left. A compatible ALF is constructed as in proof of Theorem \ref{akoz} with the exception that the $1$-handle of the cancelling pair is indicated as a pair of $3$-balls, on top of Figure \ref{fig:repairPALF}. Now we perform a left twist along the $1$-handle to repair the ALF. The square-bridge diagram and the corresponding PALF are shown on the right of Figure \ref{fig:repairsq} and  on the bottom of Figure \ref{fig:repairPALF}. Note that two diagrams in Figure \ref{fig:repairPALF} differ by a Dehn twist along the curve labeled as $\alpha$.

\begin{figure}[h]
	\includegraphics[width=0.70\textwidth]{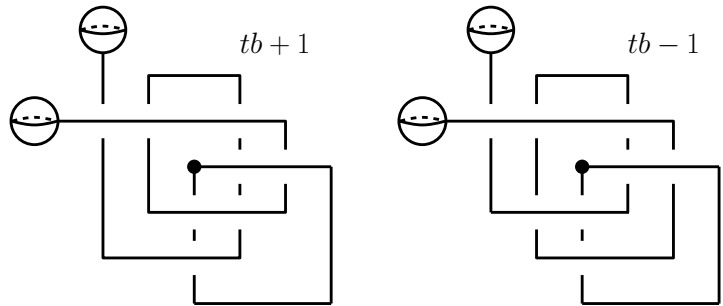}
	\caption{ALF repairing}
	\label{fig:repairsq}
\end{figure}

\newpage

\begin{figure}[h]
	\includegraphics[width=0.45	\textwidth]{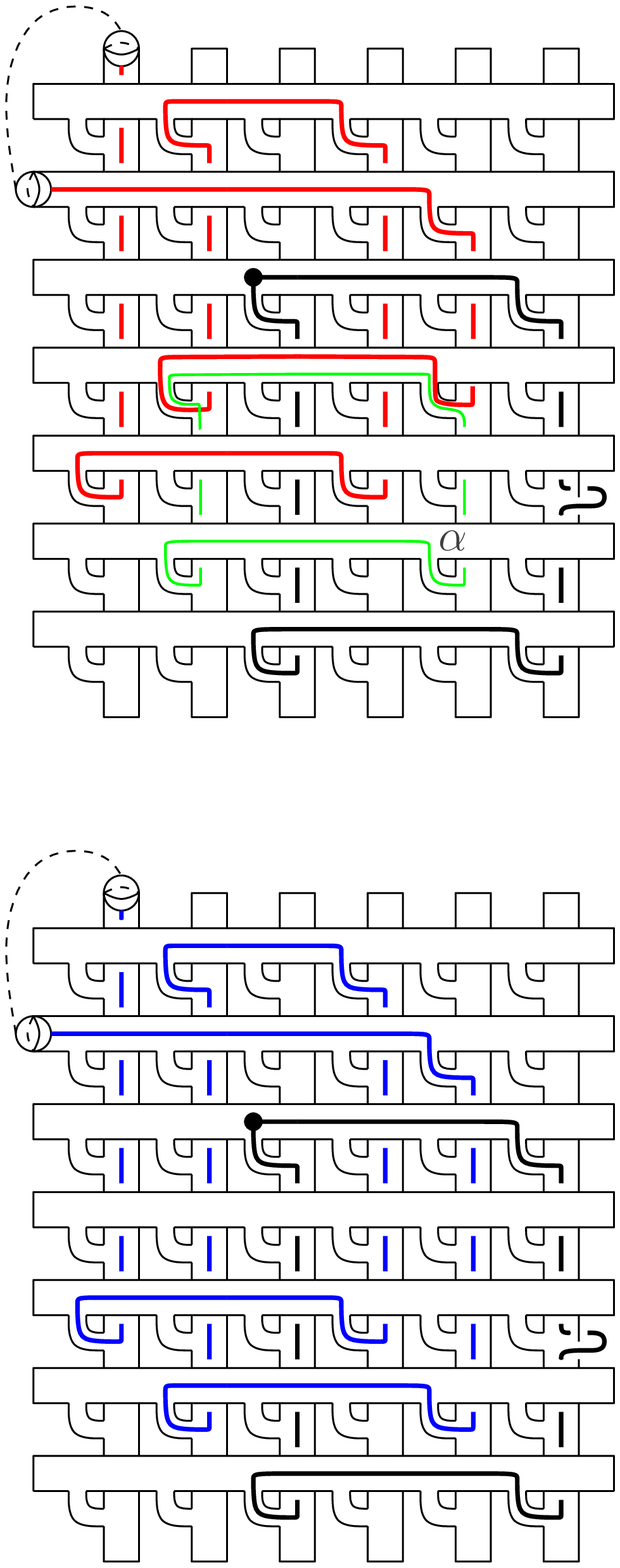}
	\caption{Alf picture}
	\label{fig:repairPALF}
\end{figure}

\newpage 

As explained in the outline, after decomposing  $X$ as a union of a convex ALF and a concave BLF such that the boundary openbooks match, for every defective 2-handle on the convex side we carefully choose a section linking the attaching circle of that  handle  twice. By carving the convex side along all such sections  and attaching the corresponding 2-handles  to the concave side, we obtain another convex-concave decomposition. Note that the boundary open books still match, but this time convex part satisfies the assumptions of the following theorem which allows us to  exchange  all the achiral singularities with Lefschetz singularities.     

\vspace{.05in}

\begin{thm}
Suppose that a $2$-handlebody $X$ admits a convex ALF constructed as in the proof of Theorem \ref{akoz}, possibly with some extra stabilizations.  Assume further that every framed link representing an achiral vanishing cycle passes over  a $1$-handle twice. Then, $X$ also admits a PALF, and there is a self diffeomorphism  of $X$ (isotopic to identity), which is a fibration equivalence between ALF and PALF structures of $X$ outside of the achiral singular points.   
\end{thm}

\end{exam}



\end{document}

%% file: goksty.tex
\newcommand{\ben}{\begin{enumerate}}
\newcommand{\een}{\end{enumerate}}
\newcommand{\be}{\begin{equation}}
\newcommand{\ee}{\end{equation}}
\newcommand{\bea}{\begin{eqnarray}}
\newcommand{\eea}{\end{eqnarray}}
\newcommand{\bc}{\begin{center}}
\newcommand{\ec}{\end{center}}

\newtheorem{thm}{Theorem}[section]
\newtheorem{cor}[thm]{Corollary}
\newtheorem{lem}[thm]{Lemma}
\newtheorem{prop}[thm]{Proposition}

\theoremstyle{definition}
\newtheorem{defn}[thm]{Definition}

\theoremstyle{remark}
\newtheorem{rem}[thm]{\rm\bfseries{Remark}}


%% file: geometricADK14.bbl
\vspace{.03in}

\begin{thebibliography}{99999}
\bibitem[A]{A} S. Akbulut,  \emph{ On $2$-dimensional homology classes of $4$-manifolds}, Math. Proc. Camb. Phil. Soc. 82 (1977), 99--106. 
\bibitem[AM]{AM} S. Akbulut and R. Matveyev, \emph{A convex decomposition theorem for $4$-manifolds}, Int. Math Res. Notices \textbf{7} (1998), 371--381. 
\bibitem[AO]{AO} S. Akbulut and B. Ozbagci, \emph{Lefschetz fibrations on compact Stein surfaces}, Geom. Topol. \textbf{5} (2001), 319--334.
\bibitem[Ar]{Ar} M. F. Arikan, \emph{On support genus of a contact structure}, Journal of G\"okova Geometry Topology, Volume \textbf{1} (2007), 96--115.
\bibitem[ADK]{ADK} D. Auroux, S. K. Donaldson and L. Katzarkov, \emph{Singular Lefschetz pencils}, Geom. Topol. \textbf{9} (2005), 1043--1114.
\bibitem[B]{B} R. I. Baykur, \emph { Existence of broken Lefschetz fibrations}, Int. Math. Res. Not. IMRN (2008) Art. ID rnn101, 15 pp.
\bibitem[B2]{B2} R. I. Baykur, \emph{Topology of broken Lefschetz fibrations and near-symplectic 4-manifolds}, to appear in Pac. J. Math. arXiv:0801.0192.
\bibitem[DG]{DG} F. Ding and H. Geiges, \emph{A Legendrian surgery presentation of contact $3$-manifolds}, Math. Proc. Cambridge Philos. Soc. \textbf{136} (1004), 583--598.
\bibitem[E]{E} Y. Eliashberg, \emph{Topological Characterization of Stein manifolds in dimension $>2$}, International J. of Math. \textbf{1} (1990), 29--46.
\bibitem[E2]{E2} Y. Eliashberg, \emph{Contact $3$-manifolds twenty years since J. Martinet's work}, Ann. Inst. Fourier, \textbf{42} (1992), 165--192.
\bibitem[E3]{E3}Y. Eliashberg, \emph{Classification of overtwisted contact structures on 3-manifolds}, Invent. Math.
\textbf{98} (1989), 623–-637.
\bibitem[EF]{EF} J. B. Etnyre and T. Fuller, \emph{Realizing $4$-manifolds as achiral Lefschetz fibrations}, Int. Math. Res.
Not. (2006) Art. ID 70272, 21.
\bibitem[EO]{EO} T. Etg\"u and B. Ozbagci, \emph{Explicit horizontal open books on some plumbings},
Internat. J. Math. \textbf{17} (2006), no. 9, 1013--1031. 
\bibitem[GK]{GK} D. Gay and R. Kirby, \emph{Constructing symplectic forms on 4-manifolds which vanish on circles}, Geom. Topol.  \textbf{8}  (2004), 743--777.  
\bibitem[GK2]{GK2} D. Gay and R. Kirby, \emph{Constructing Lefschetz-type fibrations on four-manifolds}, Geom. Topol. \textbf{11}  (2007), 2075--2115.
\bibitem[Gi]{Gi} E. Giroux, \emph{Contact geometry: from dimension three to higher dimensions}, Proceedings of the International Congress of Mathematicians (Beijing 2002), 405--414.
\bibitem[G]{G} R. Gompf, \emph{Handlebody contruction of Stein surfaces}, Ann. of Math. \textbf{148} (1998), 619--693.
\bibitem[H]{H} J. Harer, \emph{Pencils of curves on $4$-manifolds}, Ph.D thesis, UC Berkeley, 1974.
\bibitem[L]{L} Y. Lekili, \emph{Wrinkled Fibrations on Near-Symplectic Manifolds}, Geom. Topol. \textbf{13} (2009) 277-318. 
\bibitem[Ly]{Ly} H. Lyon, \emph{Torus knots in the complements of links and surfaces}, Michigan Math. J. \textbf{27} (1980), 39--46.
\bibitem[O]{O} B. Ozbagci, \emph{An open book decomposition compatible with rational contact surgery}, Proceedings of G\"okova Geometry-Topology Conference (2005), 175--186.
\bibitem[OS]{OS} B. Ozbagci and A. I. Stipsicz, \emph{Surgery on Contact 3-Manifolds and Stein Surfaces}, Bolyai Society Mathematical Studies, vol. 13, Springer, Berlin, 2004. 
\bibitem[St]{St}A. I. Stipsicz, \emph{Surgery diagrams and open book decompositions of contact 3-manifolds}, Acta Math. Hungar. \textbf{108} (2005), no. 1-2, 71--86.
\end{thebibliography}
